\documentclass[12pt]{article}
\usepackage{}

\usepackage{epsfig}
\usepackage{latexsym}
\usepackage{caption}
\usepackage{amsfonts}

\usepackage{amssymb}
\usepackage{mathrsfs}
\usepackage{amsmath}
\usepackage{enumerate}
\usepackage{graphicx}
\usepackage{MnSymbol}
\usepackage{float}
\usepackage{pict2e}
\usepackage{enumerate}
\usepackage{enumitem}

\usepackage{graphics}
\usepackage{MnSymbol}
\usepackage{tikz}
\usepackage{cite}
\usepackage{ulem}
\usepackage{extarrows}
\usepackage{cases}
\usepackage{tabu}

\usepackage{tikz-3dplot}
\usepackage{asymptote}

\numberwithin{equation}{section}

\usepackage[pdfauthor={derajan},pdftitle={How to do this},pdfstartview=XYZ,bookmarks=true,
colorlinks=true,linkcolor=blue,urlcolor=blue,citecolor=blue,pdftex,bookmarks=true,linktocpage=true,  hyperindex=true]{hyperref}

\usepackage{color}

\makeatletter

\newcommand{\Rmnum}[1]{\expandafter\@slowromancap\romannumeral #1@}

\makeatother

\begin{document}

\newtheorem{theorem}{Theorem}
\newtheorem{observation}[theorem]{Observation}
\newtheorem{corollary}[theorem]{Corollary}
\newtheorem{algorithm}[theorem]{Algorithm}
\newtheorem{definition}[theorem]{Definition}
\newtheorem{guess}[theorem]{Conjecture}
\newtheorem{claim}{Claim}
\newtheorem{problem}[theorem]{Problem}
\newtheorem{question}[theorem]{Question}
\newtheorem{lemma}[theorem]{Lemma}
\newtheorem{proposition}[theorem]{Proposition}
\newtheorem{fact}[theorem]{Fact}

\captionsetup[figure]{labelfont={bf},name={Fig.},labelsep=period}
\makeatletter
  \newcommand\figcaption{\def\@captype{figure}\caption}
  \newcommand\tabcaption{\def\@captype{table}\caption}
\makeatother

\newtheorem{acknowledgement}[theorem]{Acknowledgement}

\newtheorem{axiom}[theorem]{Axiom}
\newtheorem{case}[theorem]{Case}
\newtheorem{conclusion}[theorem]{Conclusion}
\newtheorem{condition}[theorem]{Condition}
\newtheorem{conjecture}[theorem]{Conjecture}
\newtheorem{criterion}[theorem]{Criterion}
\newtheorem{example}[theorem]{Example}
\newtheorem{exercise}[theorem]{Exercise}
\newtheorem{notation}{Notation}
\newtheorem{solution}[theorem]{Solution}
\newtheorem{summary}[theorem]{Summary}

\newenvironment{proof}{\noindent {\bf
Proof.}}{\rule{3mm}{3mm}\par\medskip}
\newcommand{\remark}{\medskip\par\noindent {\bf Remark.~~}}
\newcommand{\pp}{{\it p.}}
\newcommand{\de}{\em}
\newcommand{\mad}{\rm mad}
\newcommand{\qf}{Q({\cal F},s)}
\newcommand{\qff}{Q({\cal F}',s)}
\newcommand{\qfff}{Q({\cal F}'',s)}
\newcommand{\f}{{\cal F}}
\newcommand{\ff}{{\cal F}'}
\newcommand{\fff}{{\cal F}''}
\newcommand{\fs}{{\cal F},s}

\newcommand{\wrt}{with respect to }
\newcommand{\q}{\uppercase\expandafter{\romannumeral1}}
\newcommand{\qq}{\uppercase\expandafter{\romannumeral2}}
\newcommand{\qqq}{\uppercase\expandafter{\romannumeral3}}
\newcommand{\qqqq}{\uppercase\expandafter{\romannumeral4}}
\newcommand{\qqqqq}{\uppercase\expandafter{\romannumeral5}}
\newcommand{\qqqqqq}{\uppercase\expandafter{\romannumeral6}}

\newcommand{\qed}{\hfill\rule{0.5em}{0.809em}}

\newcommand{\var}{\vartriangle}

\title{{\large \bf Multiple list colouring of $3$-choice critical graphs}}

\author{Rongxing Xu\thanks{Department of Mathematics, Zhejiang Normal University,  China. E-mail:
xurongxing@yeah.net. Grant Number: NSFC 11871439. Supported also by
China Scholarship Council  and partially supported by the ANR project HOSIGRA (ANR-17-CE40-0022).}, \and  Xuding Zhu\thanks{Department of Mathematics, Zhejiang Normal University,  China.  E-mail: xdzhu@zjnu.edu.cn. Grant Number: NSFC 11971438. Supported also by
the 111 project of The Ministry of Education of China.}  }

\maketitle

\begin{abstract}
A graph $G$ is called $3$-choice critical if $G$ is not $2$-choosable but any proper subgraph  is $2$-choosable. A  characterization of $3$-choice critical graphs was given by Voigt in [On list Colourings and Choosability of Graphs, Habilitationsschrift, Tu Ilmenau(1998)]. Voigt conjectured that if $G$ is a bipartite $3$-choice critical graph, then $G$ is $(4m, 2m)$-choosable for every  integer $m$. This conjecture was disproved by  Meng, Puleo and Zhu in [On (4, 2)-Choosable Graphs, Journal of Graph Theory 85(2):412-428(2017)]. They  showed that if $G=\Theta_{r,s,t}$ where $r,s,t$ have the same parity and $\min\{r,s,t\} \ge 3$, or  $G=\Theta_{2,2,2,2p}$ with $p \ge 2$, then $G$ is bipartite $3$-choice critical, but   not $(4,2)$-choosable. On the other hand,
all the other bipartite 3-choice critical graphs are $(4,2)$-choosable.  
This paper strengthens the result of Meng, Puleo and Zhu and shows that all the other bipartite $3$-choice critical graphs are $(4m,2m)$-choosable for every integer $m$. 
\end{abstract}

\section{Introduction}

An {\em $a$-list assignment} of a graph $G$ is a mapping $L$ which assigns to each vertex $v$ of $G$ a set $L(v)$ of $a$ colours. A {\em $b$-fold coloring} of $G$ is a  mapping $\phi$  which assigns to each vertex $v$ of $G$ a  set $\phi(v)$ of $b$ colors such that for every edge $uv$, $\phi(u) \cap \phi(v) = \emptyset$.
An{ \em $(L,b)$-colouring} of $G$ is a $b$-fold coloring $\phi$ of $G$ such that $\phi(v) \subseteq L(v)$  for each vertex $v$.
We say $G$ is {\em $(a,b)$-choosable} if for any $a$-list assignment $L$ of $G$, there is an $(L,b)$-colouring of $G$.
We say $G$ is {\em  $a$-choosable } if $G$ is $(a,1)$-choosable. 
The concept of list colouring of graphs was introduced independently by Erd\H{o}s, Rubin and Taylor \cite{ERT} and Vizing \cite{Vizing1976} in the 1970's. Since then, list colouring of graphs has attracted considerable attention and becomes an important branch of chromatic graph theory.

Erd\H{o}s, Rubin and Taylor \cite{ERT} characterized all the $2$-choosable graphs. Given a graph $G$, the {\em core} of $G$ is obtained from $G$ by repeatedly removing degree $1$ vertices. Denote by $\Theta_{k_1,k_2,\ldots, k_q}$   the graph consisting of internally vertex disjoint paths of lengths $k_1,k_2, \ldots, k_q$ connecting two vertices $u$ and $v$. Erd\H{o}s, Rubin and Taylor proved that a graph $G$ is $2$-choosable if and only if the core of $G$ is $K_1$ or an even cycle or $\Theta_{2,2,2p}$ for some positive integer $p$. 

We say a graph $G$ is   {\em  $3$-choice-critical}  if $G$ is not $2$-choosable and any proper subgraph of $G$ is $2$-choosable.
In 1998, Voigt characterized all the $3$-choice-critical graphs.

\begin{theorem}[\cite{Voigt1998}]
\label{3-choice-critical}
	A graph is 3-choice-critical if and only if it is one of the following:
	\begin{enumerate}
		\item An odd cycle.
		\item Two vertex-disjoint even cycles joined by a path.
		\item Two even cycles with one vertex in common. 
		\item  $\Theta_{2r,2s,2t}$  with $r\geq 1$, and $s, t>1$, or  
			  $\Theta_{2r+1,2s+1,2t+1}$ with $r\ge 0$, $s,t > 0$.
		\item $\Theta_{2,2,2,2t}$ graph with $t\geq 1$.
	\end{enumerate}
\end{theorem}

Except the odd cycle, all the other $3$-choice-critical graphs  are bipartite. 
In \cite{Voigt1998}, Voigt conjectured that every bipartite $3$-choice -critical graph $G$  is $(2m,m)$-choosable for every even integer $m$. This conjecture is true if $G=\Theta_{2,2,2,2}$ \cite{TuzaVoigt1996}.
However,  Meng, Puleo and Zhu \cite{4choose2} proved that  if $min\{r,s,t\} \geq 3$,  $r,s,t$ have the same parity, then $\Theta_{r,s,t}$ is not $(4,2)$-choosable, and if $t \geq 2$, then $\Theta_{2,2,2,2t}$ is not $(4,2)$-choosable.
Nevertheless, the other bipartite 3-choice-critical graphs are $(4,2)$-choosable  \cite{4choose2}.
 It was conjectured by Erd\H{o}s, Rubin and Taylor \cite{ERT} that every $(a,b)$-choosable graph is $(am,bm)$-choosable. This conjecture was refuted recently by Dvo\v{r}\'{a}k, Hu and Sereni \cite{DHS} who proved that for any integer $k \ge 4$, there exists a $k$-choosable graph which is not $(2k,2)$-choosable. 
On the other hand, 
it was proved by Tuza and Voigt \cite{TuzaVoigt1996-2choosable} that if $G$ is $2$-choosable, then $G$ is $(2m,m)$-choosable for any positive integer $m$.
A natural question is whether all the $(4,2)$-choosable  $3$-choice critical graphs  are $(4m,2m)$-choosable for all integer $m$. In this paper, we answer this question in affirmative.
  
\begin{theorem}
	\label{thm-main}
	If $G$ consists of two vertex disjoint even cycles joined by a path or
	two   even cycles  intersecting at a single vertex, or $G = \Theta_{r,s,t}$ and  $r \leq 2$, $s,t>2$ and $r,s,t$ have the same parity, then $G$ is $(4m,2m)$-choosable for every integer $m$. 
\end{theorem}   

The {\em strong fractional choice number}  of a graph $G$ studied in \cite{JZ2019,LZ2019,Zhu2018} is defined as 
$$ch_f^*(G)= \inf\{r \in \mathbb{R}: G \text{ is $(a,b)$-choosable for any $a,b$ for which   $a/b \ge r$}\}.$$
As a consequence of Theorem \ref{thm-main}, every $(4,2)$-choosable $3$-choice critical graph $G$ has $ch_f^*(G)=2$. It remains an open problem whether every bipartite $3$-choice critical graph $G$ has $ch_f^*(G) =2$.

\section{Proof of Theorem \ref{thm-main}}

The idea of the proof of Theorem \ref{thm-main} is the following: Assume $G$ is a graph as in Theorem \ref{thm-main} and $L$ is a  $4m$-list assignment  of $G$.  Let $H$ be the set of vertices of $G$ of degree at least $3$. Then $G-H$ is the disjoint union of a family of two or three paths,  where each end vertex of these paths has exactly one neighbour in $H$ unless the path consists of a single vertex $w$, in which case $w$ has two neighbours in $H$,
and other vertices of the paths have no neighbour in $H$.

We shall assign a set of $2m$ colours to each vertex in $H$. 
Then extend this pre-colouring of $H$ to an $(L, 2m)$-colouring of the remaining vertices of $G$, that consists of two or three paths.

The extension to the paths are independent to each other.  
The question in concern becomes the following: Assume $P$ is a path  with vertices $v_1, v_2, \ldots, v_n$ in order and  $L$ is a  $4m$-list assignment on $P$. Assume $S,T$ are the $2m$-sets of colours assigned to the neighbours of $v_1$ and $v_n$ in $H$ respectively (note that the neighbours of $v_1$ and $v_n$ maybe the same, in that case, $S=T$). Under what condition, we can find an $(L,2m)$-colouring of $P$ so that the end vertices of $P$ avoid the colours from $S$ and $T$, respectively?
A sufficient condition for the existence of such  an extension  to a $2m$-fold colouring of the whole path was given in \cite{4choose2}. We shall use this condition  to show that there exists appropriate $(L,2m)$-colouring of $H$ so that the colouring can be extended to all the paths in $G-H$. This is the same idea used in \cite{4choose2}.

\begin{definition} 
	\label{def-slp1}
Assume $P$ is an $n$-vertex path with vertices $v_1, v_2, \ldots, v_n$ in order.
For a list assignment $L$ of $P$, 
Let 
\begin{equation*}
\begin{array}{rl}  
X_1 =& L(v_1),   \\
X_i =& L(v_i)-X_{i-1}, i \in \{2,3,\ldots, n\},\\
S_L(P)=&\sum_{i=1}^{n}|X_i|.
\end{array}
\end{equation*}
\end{definition}

The following lemma was proved in \cite{4choose2}.  
\begin{lemma}[ \cite{4choose2}]
	\label{42lemma}
	Let $P$ be an $n$-vertex path and let $L$ be a list assignment on $P$.
	If $|L(v_1)|$, $|L(v_n)| \geq 2m$ and  {$|L(v_i)|= 4m$} for $i \in \{2,3,\ldots, n-1\}$, then path $P$ is $(L,2m)$-colourable if and only if $S_L(P) \geq 2nm$.
\end{lemma}

\begin{definition}
	 Assume $L$ is a $4m$-list assignment on $P$ and $S,T$ are two colour sets.
	 Let  $L\ominus(S,T)$   be the list assignment obtained 
	 from $L$ by deleting all colours in $S$ from $L(v_1)$, all colours in $T$ from $L(v_n)$, and leaving all other lists unchanged.
	 The damage of $(S,T)$ with respect to $L$ and $P$ is   defined as 
	 $$dam_{L,P}(S,T)=S_L(P)-S_{L\ominus(S,T)}(P).$$ 
\end{definition}

So to prove that $P$ has an $2m$-fold $L\ominus(S,T)$-colouring, it suffices to show that 
$$S_L(P) - dam_{L,P}(S,T) \ge 2nm.$$
For this purpose, a few lemmas were proved in \cite{4choose2} that give lower bounds for $S_L(P)$ and upper bounds for $dam_{L,P}(S,T)$.

\begin{definition}[\cite{4choose2}]
	\label{def-slpAX1Xn}
	Assume $n$ is an odd integer, $P$ is an $n$-vertex path with vertices $v_1, v_2, \ldots, v_n$ in order, and $L$ is a list assignment on $P$.
Let 
\begin{equation*}
\begin{array}{rl} 
\Lambda= &\mathop{\bigcap}\limits_{x\in V(P)}L(x), \\
\hat{X}_1 =& \{c \in L(v_1)-\Lambda: \mbox{the smallest index $i$ for which $c \notin L(v_i)  $ is even} \},   \\
\hat{X}_n =& \{c \in L(v_n)-\Lambda: \mbox{the largest index $i$ for which $c \notin L(v_i)  $ is even}\}.
\end{array}
\end{equation*}
\end{definition}

\begin{lemma}[\cite{4choose2}]
	\label{slp-dam} 
	Let $L$ be a list assignment on an $n$-vertex path $P$, where $n$ is odd. For  any sets of colours $S,T$,  
	$$dam_{L,P}(S,T)=|\hat{X}_1\cap S|+|\hat{X}_n\cap T|+|\Lambda \cap(S\cup T)|.$$
\end{lemma}

\begin{lemma}[\cite{4choose2}]
	\label{first lower bound for SLP}
If $L$ is a list assignment on an $n$-vertex path $P$, where $n$ is odd and $|L(v_i)|=4m$ for all $i$, then $$S_{L}(P)\ge \max\{  2(n-1)m+ |\hat{X}_1|+|\hat{X}_n|+|\Lambda|, 2(n+1)m\}.$$
\end{lemma}

The following is a key lemma for the proof in this paper.

\begin{lemma}
	\label{main-lemma}
	 Let $m, \ell$ and $\tau$ be fixed  integers,  where $m\geq 1$, $0 \leq \ell \leq 4m$, $0 \leq \tau \leq 2m-2$, $\ell +\tau \geq 2m+2$, and both $\ell$ and $\tau$ are even. Assume $x,y$ are non-negative integers with $x+y \le \ell$. 
	  Let 
	 $$F(x,y)=\sum \binom{x}{a}\binom{y}{b}\binom{\ell-x-y}{2m-\tau-a-b},$$
 where  the summation is over all  non-negative integer pairs $(a,b)$ for which  $0 \leq a \leq x$, $0 \leq b \leq y$, $a+b \le 2m-\tau$ and  $2a+b \geq \max\{2x+y+2m+1-\ell-\tau, 2m+1-\tau\}$. Then 
 $$ F(x,y) \leq \frac{1}{2}\binom{\ell}{2m-\tau}-1.$$
\end{lemma}

Note that when $a > x$ or $b > y$, then ${x \choose a}{y \choose b} = 0$.  Also $a+b \le 2m-\tau$ and 
$2a+b \geq  2x+y+2m-\tau+1-\ell$ implies that 
$2x+y \le \ell +\tau -1 + 2a+b \le \ell  +2m-\tau-1$. 
Thus the summation can be restricted to $0 \le a \le x, 0 \le b \le y$, $a+b \leq 2m-\tau$ and $2x+y \le \ell +2m-\tau-1$.

The proof of Lemma \ref{main-lemma} will be given in next section. In the rest of the section, we will prove Theorem \ref{thm-main}. In Section \ref{sec-first-half},  we will prove the first half of Theorem \ref{thm-main}: If $G$ is a graph consisting of two edge-disjoint even cycles $E$ and $F$ connected by a path $Q$ (possibly $Q$ is a single vertex path), then $G$ is $(4m,2m)$-choosable for all positive $m$. In Section  \ref{sec-first-second}, we prove the second half of Theorem \ref{thm-main}: $\Theta_{r,s,t}$ is $(4m,2m)$-choosable if $r,s,t$ have
the same parity and $r \leq 2$, $s,t > 2$. 
 
\subsection{Proof of the first part of Theorem \ref{thm-main}}
\label{sec-first-half}

 \begin{definition}
 	\label{bad(S,S)}
 	Assume $P$ is a path and $L$ is a $4m$-list assignment for $P$, $S,T$ are two $2m$-sets of colours. We say $S$ is {bad with respect to $(L,P)$}  if $dam_{L,P}(S,S) >  S_L(P) -   2nm$.  
 \end{definition}

\begin{lemma}
	\label{half-bad}
	Assume $P$ is a path with an odd number of vertices, $L$ is a $4m$-list assignment on $P$, and $W$ is a set of $4m$ colours. Then $W$ has less than $\frac{1}{2}\binom{4m}{2m}$ bad $2m$-subsets with respect to $(L,P)$. 
\end{lemma}

\begin{proof} 
	Let $\Lambda,\hat{X}_1,\hat{X}_n$ be calculated for $P$ as in Definition \ref{def-slpAX1Xn}. Let $X=\hat{X}_1\cap \hat{X}_n\cap W $, $Y=[(\hat{X}_1\Delta\hat{X}_n)\cup \Lambda]\cap W$ ($\Delta$ means symmetric difference) and $Z=W-Y-X$. 
	
	Assume $S$ is a bad subset of $W$ with respect to $(L,P)$. 
	
 	 Let $A=S\cap X$, $B=S\cap Y$ and $C= S\cap Z$, see Figure \ref{W&S}. Let $|X|=x$, $|Y|=y$, $|Z|=z$, $|A|=a$, $|B|=b$ and $|C|=c$.

	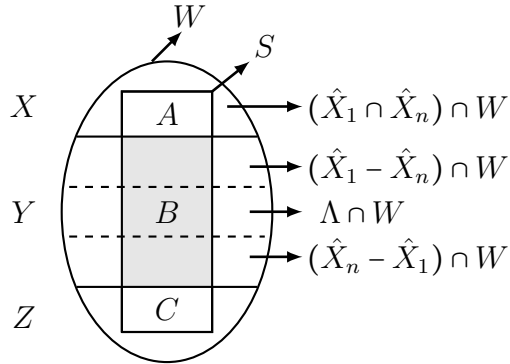
\begin{figure}[H]
	\centering
		\begin{tikzpicture}[>=latex]
		\draw [line width =0.8pt](1.4,2)ellipse(1.4 and 2);
		\draw [line width =0.8pt](0.2,1)--(2.6,1);
		\draw [line width =0.8pt](0.2,3)--(2.6,3);
				\draw [line width =0.8pt,dashed](0.1,1.67)--(2.7,1.67);
				\draw [line width =0.8pt,dashed](0.1,2.33)--(2.7,2.33);
		\filldraw [line width =0.8pt, fill opacity=0](0.8,3)--(0.8,3.6)--(2,3.6)--(2,3);
		\filldraw [line width =0.8pt,draw opacity=0, fill opacity=0.1](0.8,3)--(0.8,1)--(2,1)--(2,3)--(0.8,3);
		\filldraw [line width =0.8pt, fill opacity=0](0.8,1)--(0.8,0.4)--(2,0.4)--(2,1);
		\filldraw [line width =0.8pt,draw opacity=1, fill opacity=0](0.8,0.4)--(0.8,3.6)--(2,3.6)--(2,0.4)--(0.8,0.4);\
		\draw [line width =1pt,->](2.2,3.4)--(3.2,3.4);
			\draw [line width =1pt,->](2,3.6)--(2.5,4);
		\draw [line width =1pt,->](2.5,2.6)--(3.2,2.6);
		\draw [line width =1pt,->](2.5,2)--(3.2,2);
		\draw [line width =1pt,->](2.5,1.4)--(3.2,1.4);
		\draw [line width =1pt,->](1.2,4)--(1.6,4.4);
	
		\node at (4.6,3.4){$(\hat{X}_1\cap\hat{X}_n)\cap W$};
		\node at (4.6,2.6){$(\hat{X}_1-\hat{X}_n)\cap W$};
		\node at (4,2){$\Lambda\cap W$};
		\node at (4.6,1.4){$(\hat{X}_n-\hat{X}_1)\cap W$};
			\node at (-0.5,3.4){$X$};
			\node at (-0.5,2){$Y$};
			\node at (-0.5,0.6){$Z$};
				\node  at (1.4,3.3){$A$};
				\node  at (1.4,2){$B$};
				\node  at (1.4,0.7){$C$};
		\node at (1.7,4.6){$W$};
		\node at (2.7,4.2){$S$};	
		\end{tikzpicture}
		\caption{$W, X,Y,Z$ and $S, A, B,C$}
		\label{W&S}
	\end{figure}

Since $W$ is the disjoint union of $X, Y,Z$ and $S$ is the disjoint union of $A,B,C$, we have 
\begin{eqnarray*}
	 x+y+z=4m, \ \ a+b+c =2m.
\end{eqnarray*}
 
As 
	\begin{equation*}
	\begin{array}{lll}
	|\hat{X}_1|+|\hat{X}_n|+|\Lambda| & =    & |\hat{X}_1 \cup \hat{X}_n|+|\hat{X}_1 \cap \hat{X}_n|+|\Lambda|\\
	& \geq & |(\hat{X}_1 \cup \hat{X}_n)\cap W|+|(\hat{X}_1\cap \hat{X}_n)\cap W|+|\Lambda \cap W|\\
	& = & |(\hat{X}_1 \Delta\hat{X}_n)\cap W|+2|(\hat{X}_1\cap \hat{X}_n)\cap W|+|\Lambda \cap W|\\
	&  =   & 2x+ y,
	\end{array}
	\end{equation*}
by   Lemma \ref{first lower bound for SLP}, we have $S_{L}(P)  \geq \max\{  2x+ y+2(n-1)m, 2(n+1)m\}$.	
By Lemma \ref{slp-dam},   	$$dam_{L,P}(S,S) = |\hat{X}_1\cap S|+|\hat{X}_n\cap S|+|\Lambda \cap S|
=|(\hat{X}_1 \Delta\hat{X}_n)\cap S|+2| (\hat{X}_1 \cap \hat{X}_n ) \cap S|+|\Lambda \cap S|=2a+b.$$	
As $
dam_{L,P}(S,S) > S_{L}(P) -2nm$, we conclude that 
\begin{equation*}
\label{2b+c geq max}
2a+b  \geq  \max\{  2x+ y-2m+1, 2m+1\}.
\end{equation*}

Thus the number of bad $2m$-subsets of $W$ with respect to $(L,P)$ is at most

 \begin{equation*}
  F(x,y)= \sum \binom{x}{a}\binom{y}{b}\binom{4m-x-y}{2m-a-b},
  \end{equation*}
  where the summation is over all pairs of integers $(a,b)$ with  $0 \leq a \leq x$, $0 \leq b \leq y$, $a+b \leq 2m$ and 
  $2a+b \geq \max\{2x+y-2m+1,2m+1\}$.  By the special case with $\ell =4m$ and $\tau=0$ of Lemma \ref{main-lemma},  $F(x,y) < \frac{1}{2}\binom{4m}{2m}$.  
   
This completes the proof of Lemma \ref{half-bad}.
\end{proof}

If $G$ consists of two even cycles intersecting at a single vertex $v$, then $G-v$ is consists of two odd paths $P_1$ and $P_2$. It follows from Lemma \ref{half-bad} that there is a $2m$-subset $S$ of $L(v)$ such that $S$ is not bad with respect to both $(L,P_1)$ and $(L,P_2)$. Thus we can colour $v$ by $S$ and extend it to an $(L, 2m)$-colouring of $G$.

Assume $G$ consists of two even cycles $E$ and $F$  joined by a path $Q$, and let $u,v$ be the end vertices of $Q$, as shown in
Figure \ref{decomposing}.
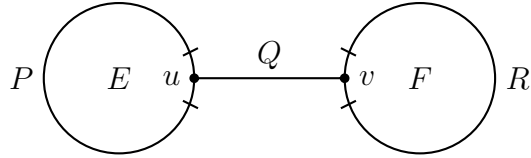
\begin{figure}[H]{}
\centering
	\begin{tikzpicture}	
	\draw [line width =0.8pt](1,1)circle (1);
	\filldraw [line width =0.8pt](2,1)circle (0.05)--(4,1)circle (0.05);
		\filldraw [line width =0.8pt](1.85,1.3)--(2.05,1.4);
		\filldraw [line width =0.8pt](1.85,0.7)--(2.05,0.6);
		\filldraw [line width =0.8pt](4.15,0.7)--(3.95,0.6);
		\filldraw [line width =0.8pt](4.15,1.3)--(3.95,1.4);
	\draw [line width =0.8pt](5,1)circle (1);
	\node at (3,1.3){$Q$};
    \node at (1,1){$E$};
    \node at (5,1){$F$};
    \node at (1.7,1){$u$};
    \node at (4.3,1){$v$};
	\node at (-0.3,1){$P$};
	\node at (6.3,1){$R$};
\end{tikzpicture}
\caption{Decomposing $G$ into $P,Q,R$.}
\label{decomposing}
\end{figure}

Observe that there is an injective function $h:\binom{L(u)}{2m} \rightarrow \binom{L(v)}{2m}$ such that for all $S \in \binom{L(u)}{2m}$, the precolouring $\phi(u)=S$, $\phi(v)=h(S)$ extends to all of $Q$.   Indeed, if $Q$ consists of a single vertex $v$, then $u=v$ and $h(S)=S$.
Otherwise,
for each $S \in \binom{L(u)}{2m}$, let $\phi(u) = S$, extend $\phi$ to a $2m$-fold $L$-colouring $\phi$ of $Q$. We simply let $h(S) = \phi(v)$.

By Lemma \ref{half-bad},  $L(u)$ has less than $\frac{1}{2}\binom{4m}{2m}$ bad $2m$-subsets  with respect to $(L,P)$, and $L(v)$ has less than $\frac{1}{2}\binom{4m}{2m}$ bad subsets of size $2m$ with respect to $(L,R)$. So there exists some $S$ such that   $S$ is not bad with respect to $(L,P)$ and $h(S)$ is not bad with respect to $(L,R)$. Therefore the pre-colouring $\phi$ of $u,v$ defined as $\phi(u)=S$ and $\phi(v)=h(S)$ extends to an $(L,2m)$-colouring of $G$.   This completes the first half of Theorem \ref{thm-main}.
\qed-

\subsection{Proof of the second part of Theorem \ref{thm-main}}
\label{sec-first-second}

The following lemma was proved in \cite{4choose2}.

\begin{lemma}[\cite{4choose2}]
	Let $G$ be a graph, let $v \in V(G)$, and let $G'$ be obtained from G by deleting $v$ and merging its neighbours. If $G$ is $(4m,2m)$-choosable, then $G'$ is $(4m,2m)$-choosable.  
\end{lemma}

\begin{corollary}
	If $\Theta_{2,2r,2s}$ is $(4m,2m)$-choosable, then $\Theta_{1,2r-1,2s-1}$ is $(4m,2m)$-choosable. 
\end{corollary}

So it suffices to show that $\Theta_{2,2r,2s}$ is $(4m,2m)$-choosable, where $r, s \ge 2$. Instead of proving it directly, we prove the following stronger result.

\begin{theorem}
	\label{main-second-part-stronger}
Assume $G=\Theta_{2,2r,2s}$, where $r, s \geq 2$. Let $u,v$ be the two vertices of degree $3$. Let $P^0,P^1,P^2$ be the three paths of $G-\{u,v\}$.
Assume   $P^i=(v^i_1,v^i_2,\ldots,v^i_{n_i} )$, $|V(P^0)|=1$,   $v^i_1$ is adjacent to $u$ and $v^i_{n_i}$ is adjacent to $v$. 
	 Assume $\ell, \tau$ are non-negative even integers and  $L$ is a list assignment for $G$ satisfying the following:
	 \begin{enumerate}[label= {(C\arabic*)}]
	 	\item \label{C1}   $\tau \leq 2m$ and $\ell +\tau \geq 2m$.
	 	\item \label{C2} $|L(u)|=|L(v)| = \ell \ge 0$.
	 	\item \label{C3} For each $i \in \{0,1,2\}$, $|L(v^i_1)|,|L(v^i_{n_i})|\geq 4m-\tau$.
	 	\item \label{C4}  $|L(w)|= 4m$ for $w \ne u,v, v^i_1, v^i_{n_i}$.
	 	\item \label{C5}  For $i=0,1,2$, $S_L(P^i) \geq 2m|V(P^i)|+2m-\tau$, and   $$dam_{L,P^i}(L(u),L(v)) \leq S_L(P^i)-2m|V(P^i)|+\ell-2m+\tau.$$
	 \end{enumerate}	
	 	 Then there exists a   set $S \subset L(u)$ and a   set $T \subset L(v)$ satisfying $|S|=|T|=2m-\tau$ such that for each $i$, $$dam_{L,P^i}(S,T) \leq S_L(P^i)-2m|V(P^i)|.$$
\end{theorem}

 Theorem \ref{thm-main} follows from Theorem \ref{main-second-part-stronger} by setting $\ell =4m$ and $\tau=0$. Indeed, by Lemma \ref{first lower bound for SLP}, $S_{L}(P^i) \geq 2m|V(P^i)|+2m$, and $S_{L}(P) -2m|V(P^i)|+2m \geq |\hat{X}^i_1|+|\hat{X}^i_n|+|\Lambda^i| \geq dam_{L,P^i}(L(u),L(v))$ (The last inequality holds by Lemma \ref{slp-dam}). So there exist two   sets   $S \subset L(u)$, $T \subset L(v)$ such that $|S|=|T|=2m$ and $dam_{L,P^i}(S,T) \leq S_L(P^i) -2m|V(P^i)|$,  which implies that $G$ is $(4m,2m)$-choosable.

 Let $L(u)=\{c_1, c_2, \ldots, c_{\ell} \}$ and $L(v)=\{c'_1,c'_2, \ldots, c'_{\ell}\}$ be indexed in such a way   that $c_j=c'_j$ whenever $c_j \in L(u)\cap L(v)$. In other words, $\{c_i,c'_i\}\cap \{c_j,c'_j\}=\emptyset$ whenever $i\neq j$.  

\begin{definition}
	\label{bad simple pair}
	For a fixed indexing of $L(u)$ and $L(v)$, a {\rm couple} is a tuple of the form $(c_j,c'_j)$ for $j\in\{1,2,\ldots,\ell \}$. When we write a couple, we suppress the parentheses and simply write $c_jc'_j$. A {\rm pair} is a tuple $(S,T)$ with $S\subset L(u)$, $T\subset L(v)$, and  $|S|=|T|$, and we define the \emph{size} of a pair as $|S|$. 
	A  pair $(S,T)$ is {\rm bad with respect to $(L,P)$} if $dam_{L,P}(S,T) > S_{L}(P) - 2m|V(P)|$. 
 {A {\rm simple pair} is a pair $(S,T)$ such that $S,T$ have the same index set.}
\end{definition}
 
By Lemma \ref{slp-dam}, we know that if $(S,T)$ is a simple pair, then 
\begin{equation}
	\label{comput for dam(S,T)}
	dam_{L,P}(S,T)=\sum_{c_j\in S}dam_{L,P}(\{c_j\},\{c'_j\}).
\end{equation}

In the following, we may write $dam_{L,P}(c,c')$ for $dam_{L,P}(\{c\},\{c'\})$.
The following observation follows from Lemma \ref{slp-dam}.

\begin{observation}
	\label{obs-couple}
	 For any couple $cc'$, the following hold:
	  \begin{enumerate}
	  	\item $dam_{L,P}(c,c') = 2$ if   $c \in \hat{X}_1 \cup \Lambda$ and $c' \in \hat{X}_n \cup \Lambda$, and moreover if $c = c'$, then $c \notin \Lambda$; 
	  	\item $dam_{L,P}(c,c') = 1$ if $c \in \hat{X}_1 \cup \Lambda$ or  $c' \in \hat{X}_n \cup \Lambda$ but not both  unless $c=c' \in \Lambda$;  
	  	\item $dam_{L,P}(c,c') = 0$ if $c \notin \hat{X}_1 \cup \Lambda$ and $c' \notin \hat{X}_n \cup \Lambda$.
	  \end{enumerate}
\end{observation}

\begin{definition}
	\label{safe-light-heavy}
	Assume $c_jc'_j$ is a couple. 
	\begin{itemize}
		\item   $c_jc'_j$ is {\rm heavy} for the internal path $P$ if $dam_{L,P}(c_j, c'_j)=2$;
		\item   $c_jc'_j$ is {\rm light} for the internal path $P$ if $dam_{L,P}(c_j, c'_j)=1$;
		\item   $c_jc'_j$ is {\rm safe} for the internal path $P$ if $dam_{L,P}(c_j, c'_j)=0$.
	\end{itemize}
\end{definition}

For each $i\in \{0,1,2\}$, let $x^{(i)},y^{(i)},z^{(i)}$ denote the number of heavy, light couples and safe couples for $P^{i}$, respectively. For a simple pair $(S,T)$ which is of size $2m-\tau$,  let $a^{(i)}(S,T),b^{(i)}(S,T),c^{(i)}(S,T)$ denote the number of heavy, light and safe couples for $P^{i}$ in  $ (S,T)$, respectively. 

It follows from the definition that $x^{(i)}+y^{(i)}+z^{(i)}=\ell$, and  $a^{(i)}(S,T)+b^{(i)}(S,T)+c^{(i)}(S,T)=2m-\tau$. Thus by Equality (\ref{comput for dam(S,T)}), $dam_{L,P^i}(S,T) = 2a^{(i)}(S,T)+b^{(i)}(S,T)$. Let $\beta(P^{i})$ denote the   number of  bad simple pairs of size $2m-\tau$ with respect to  $(L, P^i)$. We write $\hat{X}^i_1, \hat{X}^i_n$ and $\Lambda^i$ for the sets $\hat{X}_1, \hat{X}_n, \Lambda$   calculated for $P^i$.

\medskip
\noindent 
{\bfseries Proof of Theorem \ref{main-second-part-stronger}:} 
First we observe that the conclusion of 
Theorem \ref{main-second-part-stronger} is equivalent to the statement that  there exists a pair $(S,T)$ which  is not a bad pair for any of the paths $P^0,P^1,P^2$.

The proof is by induction on $2\ell+\tau$.  First assume that $2\ell + \tau =2m$. Since both $\ell$ and $\tau$ are non-negative, and $\ell +\tau \geq 2m$, we have that $\ell = 0$ and $\tau= 2m$. By assumption, for each $i \in \{0,1,2\}$, 
$$S_L(P^i)-2m|V(P^i)| \geq dam_{L,P^i}(L(u),L(v))=0,$$ so $S_L(P^i) \geq 2m|V(P^i)|$. Then let $S=L(u)= \emptyset$, $T=L(v) =\emptyset$, and we are done. This finishes the basic step of the induction.

Thus in the sequel, we assume that  $2\ell+\tau \geq 2m+2$. If $\ell + \tau =2m$, then we let $S=L(u)$, $T=L(v)$, and we are done. Hence we assume that $\ell +\tau \geq 2m+2$. If $\tau =2m$, then the statement holds obviously, since we can just take $S=T=\emptyset$ and we are done. So we assume that $\tau \leq 2m-2$. Assume   that Theorem \ref{main-second-part-stronger} is not true for $L$.

\begin{claim}
	\label{all at least light}
	 There does not exist a simple pair $(D_u,D_v)$ such that $|D_u|=|D_v|=d \le \ell -2m+\tau$ is even,  and $dam_{L,P^i}(D_u,D_v) \geq d$ for each  $i \in \{0,1,2\}$. 
\end{claim}

\begin{proof}
	Assume $(D_u,D_v)$ is such a simple pair. Let $L'$ be a new list assignment for $G$ with $L'(u)= L(u)-D_u$, $L'(v)=L(v)-D_v$, $L'(w)=L(w)$ for $w \in V(G) \setminus \{u,v\}$. 
	
	\ref{C1}-\ref{C4} of Theorem \ref{main-second-part-stronger} are easily seen to be satisfied by $L'$, with $ \ell'= \ell - d$ and $\tau'=\tau$.   As
	\begin{equation*}
	\begin{array}{ll}
	dam_{L',P^i}(L'(u),L'(v)) & = dam_{L,P^i}(L(u),L(v))-dam_{L,P^i}(D_u,D_v) \\
	& \leq  S_L(P^i)-2m|V(P^i)|+\ell-2m+\tau-d \\
	& = S_{L'}(P^i)-2m|V(P^i)| + \ell'+ \tau'-2m,
	\end{array}
	\end{equation*}
	\ref{C5} is also satisfied by $L'$.
	By induction,   there exists a pair $(S,T)$, where $|S|=|T|=2m-\tau$ such that for each $i \in \{0,1,2\}$, $$dam_{L,P^i}(S,T) \leq S_L(P^i)-2m|V(P^i)|.$$  This completes the proof of this  claim.
\end{proof}

\begin{claim}
	\label{all at most light}
	There does not exist a simple pair $(D_u,D_v)$ such that $|D_u|=|D_v|=d \le 2m-\tau$ is even,  and $dam_{L,P^i}(D_u,D_v) \leq d$ for each  $i \in \{0,1,2\}$. 
\end{claim} 

\begin{proof}
	Assume   $(D_u,D_v)$ is such a simple pair. Let $L'$ be a new list assignment for $G$ with $L'(u)=L(u)-D_u$, $L'(v)=L(v)-D_v$, for $i=1,2$,   $L'(v^i_1)= L(v^i_1)-D_u$, $L'(v^i_{n_i})=L(v^i_{n_i})-D_v$,  $L'(v^i_j)=L(v^i_j)$ where $1 < j < n_i$, and $L'(v^0_1)=L(v^0_1)-D_u \cup D_v$. 
	
	Note that $dam_{L,P^0}(D_u,D_v) \leq d$ implies that $|L(v^0_1)-D_u \cup D_v| \ge |L(v^0_1)|-d$. So 
	\ref{C1}-\ref{C4} of Theorem \ref{main-second-part-stronger} are  satisfied by $L'$, with $ \ell'= \ell - d$ and $\tau'=\tau+d$.   As
	$S_{L}(P^i) = S_{L'}(P) + dam_{L,P^i}(D_u,D_v)$,  we have $S_L'(P^i) = S_{L}(P^i)- dam_{L,P^i}(D_u,D_v) \geq 2m|V(P^i)|+2m-\tau-d$, and also by the second part of \ref{C5}, 
	it follows that 
	\begin{equation*}
	\begin{array}{ll}
	dam_{L',P^i}(L'(u),L'(v)) & \le dam_{L,P^i}(L(u),L(v))-d \\
	& \leq  S_L(P^i)-2m|V(P^i)|+\ell-2m+\tau-d \\
	& =  S_{L'}(P^i)+d-2m|V(P^i)|+\ell-2m+\tau-d \\
	& =  S_{L'}(P^i)-2m|V(P^i)|+(\ell-d)+(\tau+d)-2m.
	\end{array}
	\end{equation*}	
	So	\ref{C5} is also satisfied by $L'$.
	By induction,   there exists a pair $(S',T')$, where $|S'|=|T'|=2m-\tau' = 2m-\tau-d$ such that for every $i$, $$dam_{L',P^i}(S',T') \leq S_{L'}(P^i)-2m|V(P^i)|.$$  Let $S=S' \cup D_u$ and  $T=T' \cup D_v$. We have $|S|=|T|=2m-\tau$ and 
	\begin{equation*}
		\begin{array}{ll}
			dam_{L,P^i}(S,T) & = dam_{L,P^i}(D_u,D_v)+dam_{L,P^i}(S',T') \\
							 & \leq dam_{L,P^i}(D_u,D_v)+  S_{L'}(P^i)-2m|V(P^i)|\\
							 & = S_{L}(P^i)-2m|V(P^i)|.
		\end{array}
	\end{equation*} 
	
	This completes the proof of Claim \ref{all at most light}.
\end{proof}

The following claim gives a necessary condition for a simple pair of size $2m-\tau$ being bad with respect to $(L,P^i)$.

\begin{claim}
	\label{claim-bad simple pair}
	If $(S,T)$ is a bad simple pair of size $2m-\tau$ with respect to $(L, P^i)$, then 
	$dam_{L,P^i}(S,T) = 2a^{(i)}(S,T)+b^{(i)}(S,T) \geq \max\{2x^{(i)}+y^{(i)}+2m+1-\ell -\tau, 2m+1-\tau\}$.
\end{claim}

\begin{proof}
	By Equality (\ref{comput for dam(S,T)}), $dam_{L,P^i}(L(u),L(v))= 2x^{(i)}+y^{(i)}$. By \ref{C5},    $$S_L(P^i) \geq \max\{2m|V(P^i)|+2x^{(i)}+y^{(i)}+2m-\ell-\tau, 2m|V(P^i)|+ 2m-\tau\}.$$  If $(S,T)$ is a bad simple pair of size $2m-\tau$ with respect to $(L, P^i)$, then by Definition \ref{bad simple pair} and above inequality,
	\begin{equation*}
	\begin{array}{ll} 
	dam_{L,P^i}(S,T) & =  2a^{(i)}(S,T)+b^{(i)}(S,T) \\
	& \geq  S_L(P^i) -2m|V(P^i)| + 1  \\
	& \ge \max\{2x^{(i)}+y^{(i)} +2m+1-\ell-\tau, 2m+1-\tau\}.
	\end{array}
	\end{equation*} 
	
	Thus we proved this claim
\end{proof}

The following claim gives an upper bound and a lower bound of the number of bad simple pairs of size $2m-\tau$ with respect to $(L,P^i)$.

\begin{claim}
	\label{second case half lemma}
	For each $i \in \{0,1,2\}$, $2 \le  \beta(P^i)  \leq \frac{1}{2}\binom{\ell}{2m-\tau}-1$.  
\end{claim}

\begin{proof}
	If a simple pair $(S,T)$ of size $2m-\tau$ is bad with respect to $(L,P^i)$, then by Claim \ref{claim-bad simple pair}, $dam_{L,P^i}(S,T)  \geq \max\{2x^{(i)}+y^{(i)} +2m+1-\ell-\tau, 2m+1-\tau\}$. Note that $a^{(i)}(S,T) + b^{(i)}(S,T) +c^{(i)}(S,T) =2m-\tau$, it follows from   Lemma \ref{main-lemma} that $\beta(P^i) \leq  \frac{1}{2}\binom{\ell}{2m-\tau}-1$. 
	
	If $\beta(P^i) \le 1$ for some $i$,   then $\beta(P^0)+\beta(P^1)+\beta(P^2)  \le  \binom{\ell}{2m-\tau}-1$. So there exists a simple pair $(S,T)$ of size $2m-\tau$ which is not bad with respect to any $(L,P^i)$, a contradiction to the assumption. 
\end{proof}

	\begin{claim}
		\label{2x+yupper bound and xz at least one}
		For each $i \in \{0,1,2\}$, $2x^{(i)}+y^{(i)} \leq \ell + 2m-\tau-1$, and $x^{(i)},z^{(i)} \geq 1$.
	\end{claim}
	\begin{proof}
		If $S_{L}(P^i) \geq 2m|V(P^i)|+2m-\tau$, then for any simple pair $(S,T)$ of size $2m-\tau$, 
		$S_L(P^i) - dam_{L,P^i}(S,T) \ge 2m|V(P^i)|$ (as  $dam_{L,P^{i}}(S,T) \leq 2m-\tau$), and hence $(S,T)$ is not bad with respect to $(L,P^i)$ which means that $\beta(P^i)=0$, a contradiction.  Thus we may assume that  $S_{L}(P^i) \leq 2m|V(P^i)|+4m-1-2\tau$. By \ref{C5},		
		\begin{equation*}
		\label{2x+y-upperbound-6m-1}
		2x^{(i)}+y^{(i)} =dam_{L,P^i} (L(u),L(v)) \leq S_{L}(P^i) - 2m|V(P^i)|+ \ell+\tau - 2m \leq \ell+2m-1-\tau.
		\end{equation*}
		
		Assume $x^{(i)}=0$ for some $i \in \{0,1,2\}$,   say $x^{(0)}=0$,  then   for every simple pair $(S,T)$ of size $2m-\tau$, $dam_{L,P^0}(S,T) \leq 2m-\tau$. As $S_L(P^i)-2m|V(P^i)| \geq 2m-\tau$ (by \ref{C5}),  $(S,T)$ is not bad with respect to $(L,P^i)$, so $\beta(P^i)=0$, in contrary to Claim \ref{second case half lemma}.  Thus $x^{(i)} \geq 1$.
		
		Assume $z^{(0)}=0$, then $x^{(0)}+y^{(0)}=\ell$ and for any simple pair $(S,T)$ of size $2m-\tau$, $a^{(0)}(S,T)+b^{(0)}(S,T)=2m-\tau$. By Claim \ref{claim-bad simple pair}, we have 
		\begin{equation*}
		\begin{array}{ll} 
		2a^{(0)}(S,T)+b^{(0)}(S,T) & = a^{(0)}(S,T) +2m-\tau \\
		& \geq 2x^{(0)}+y^{(0)} +2m+1-\ell-\tau \\
		& = x^{(0)} + \ell +2m+1-\ell-\tau \\
		& = x^{(0)}+1 +2m-\tau.
		\end{array}
		\end{equation*}
		This implies that $a^{(0)}(S,T) \geq x^{(0)}+1$, in contrary to the fact that  $a^{(0)}(S,T) \leq x^{(0)}$. 
	\end{proof}

	\begin{claim} 
		\label{evey couple is HSL} 
		Every couple is heavy (respectively, safe) for at most one internal path. 
		There is at most one couple which is light for all internal paths. If there exists a couple which is light for at least two internal paths, then it is light for all internal paths.
	\end{claim}

\begin{proof}
	Assume to the contrary,   $c_jc'_j$ is heavy for two paths, say for both $P_0$ and $P_1$. If $c_jc'_j$ is also heavy for $P^2$, then for any other couple $c_kc'_k$, we know that $(\{c_j,c_k\},\{c'_j,c'_k\})$ is a simple pair of size $2$ contradicting to Claim \ref{all at least light}. Thus $c_jc'_j$ is not heavy for $P^2$. Note that $x^{(2)} \geq 1$, there exists a couple $c_kc'_k$ which is heavy for $P^2$. It follows that $(\{c_j,c_k\},\{c'_j,c'_k\})$ is a simple pair of size $2$ contradicting to Claim \ref{all at least light}. 
	Similarly, we can prove that no couple is safe for at least two internal paths.

  If there are  two couples which are light for all internal paths, then two such couples comprise a simple pair of size $2$ which contradicts  to Claim \ref{all at most light}. 
  
  Assume the last sentence of this claim is not true, say $c_jc'_j$ is light for $P^0$ and $P^1$ but not light for $P^2$. Note that by Claim \ref{2x+yupper bound and xz at least one}, $z^{(2)} \geq 1$, so if  $c_jc'_j$ is heavy for $P^2$, then  there exists a distinct couple $c_kc'_k$ which is safe for $P^2$. By the first part of this claim, $c_kc'_k$ is safe for neither $P^0$ nor $P^1$. Then  $(\{c_j,c_k\},\{c'_j,c'_k\})$ is a simple pair of size $2$ contradicting to  Claim \ref{all at least light}. If $c_jc'_j$ is safe for $P^2$, then since $x^{(2)} \geq 1$ (by Claim \ref{2x+yupper bound and xz at least one}), there exists a distinct couple $c_kc'_k$ which is heavy for $P^2$. By the first part of this claim,  $c_kc'_k$ is heavy for neither $P^0$ nor $P^1$. Then  $(\{c_j,c_k\},\{c'_j,c'_k\})$ is a simple pair  of size $2$  contradicting to  Claim \ref{all at most light}.  
	This completes the proof of Claim \ref{evey couple is HSL}. 
\end{proof}

	\medskip
	
	Without loss of generality, we may assume that $c_0c'_0$ is heavy for $P^0$, light for $P^1$ and safe for $P^2$.  
	
	\begin{claim}
		\label{LuLv-HSL}
	For any couple $c_jc'_j$, 
		\begin{itemize}
			\item if it is heavy for $P^0$, then it is light for $P^1$, safe for $P^2$;
			\item if it is light for $P^0$, then it is safe for $P^1$, heavy for $P^2$;
			\item if it is safe for $P^0$, then it is heavy for $P^1$, light for $P^2$.
		\end{itemize}
	Consequently, $x^{(0)}=y^{(1)}=z^{(2)}$, $y^{(0)}=z^{(1)}=x^{(2)}$ and $z^{(0)}=x^{(1)}=y^{(2)}$. 
	\end{claim}
\begin{proof}	
	If $c_jc'_j$ is safe for $P^0$, 
	then  $c_jc'_j$   is light for $P^2$,  for otherwise by Claim \ref{evey couple is HSL},  $c_jc'_j$   is heavy for $P^2$, and light for $P^1$. Then  $(\{c_0,c_j\},\{c'_0,c'_j\})$ is a  simple pair of size $2$ which contradicts to Claim \ref{all at most light}. 	By Claim \ref{evey couple is HSL}, 
	$c_jc'_j$ is heavy for $P^1$. 
	
	By Claim \ref{2x+yupper bound and xz at least one}, $z^{(0)} \geq 1$, there is a couple $c_ic'_i$, which  is   safe for $P^0$. Hence $c_ic'_i$   is light for $P^2$ and  {heavy} for $P^1$.

 Also by Claim \ref{2x+yupper bound and xz at least one},  $z^{(1)} \geq 1$,   there exists a couple $c_kc'_k$ which is safe for $P^1$.
  Then $c_kc'_k$   is light for $P^0$,   otherwise by Claim \ref{evey couple is HSL},  $c_kc'_k$   is heavy for $P^0$, and light for $P^2$. Then  $(\{c_i,c_k\},\{c'_i,c'_k\})$ is a  simple pair of size $2$ which contradicts to Claim \ref{all at most light}. 
  By Claim \ref{evey couple is HSL}, 
  $c_kc'_k$ is heavy for $P^2$. 

 If  $c_jc'_j$ is heavy for $P^0$, then it is light for $P^1$, for otherwise, it is safe for $P^1$, light for $P^2$, but then $(\{c_i,c_j\},\{c'_i,c'_j\})$ is a  simple pair of size $2$ which contradicts to Claim \ref{all at most light}. 
 
	Next we show that no couple is light for all internal paths. 
	Assume that there exist a couple $cc'$ which is light for all internal paths. If $\tau  \leq 2m-4$, then $2m-\tau \geq 4$, so $(\{c_0,c_i,c_k,c\},\{c'_0,c'_i,c'_k,c'\})$ is a simple pair of size $4$   contradicting to Claim \ref{all at most light}.
	
	Assume $\tau =2m-2$. Recall that $\ell+\tau \geq 2m+2$, so if $\ell \neq 4$, then $\ell \geq 6$ and $\ell-2m+\tau \geq 4$. Thus $(\{c_i,c_j,c_k,c\},\{c'_i,c'_j,c'_k,c'\})$ is a simple pair of size $4$ which contradicts to Claim \ref{all at least light}. 
	
	Assume $\tau =2m-2$ and $\ell=4$.   Since $|V(P^0)|=1$, we know that $\hat{X}^0_1=\hat{X}^0_n=\emptyset$. By Observation \ref{obs-couple}, $c_0,c'_0 \in \Lambda^0$ and  $c_0 \neq c'_0$. As $c_0c'_0$ is light for $P^1$, by Observation \ref{obs-couple}, $c_0 \notin$ $\hat{X}^1_1 \cup \Lambda^1$ or $c'_0 \notin \hat{X}^1_n \cup \Lambda^1$.  By symmetric, we may assume that $c_0 \notin$ $\hat{X}^1_1 \cup \Lambda^1$. Then for $S=\{c_0,c\}$ and $T=\{c'_i,c'\}$, the conclusion of Theorem \ref{main-second-part-stronger} holds. 
	
	If $c_jc'_j$ is light for $P^0$,   as   $c_jc'_j$ is not light for all internal paths,  we conclude that  it must be safe for $P^1$, for otherwise,
	$c_jc'_j$ is heavy for $P^1$ and  $(\{c_k,c_j\},\{c'_k,c'_j\})$ is a  simple pair of size $2$  contradicting to Claim \ref{all at most light}. This implies that $c_jc'_j$ heavy for $P^2.$ 
\end{proof}

	\begin{claim}
		For each $i \in \{0,1,2\}$, $x^{(i)},y^{(i)},z^{(i)} \geq2$.
	\end{claim}
	
	\begin{proof}
	Assume to the contrary that $x^{(0)}=1$.
	As $x^{(0)}+y^{(0)}+z^{(0)} = \ell$, we have
	 $z^{(0)}=\ell -y^{(0)}-1$.   
	 By Claim \ref{2x+yupper bound and xz at least one} and Claim \ref{LuLv-HSL}, $2y^{(0)}+z^{(0)}=2x^{(2)}+y^{(2)} \leq \ell +2m-\tau-1$. So $2y^{(0)}+(\ell -1 -y^{(0)}) \leq \ell +2m-\tau -1$,  which implies that $y^{(0)} \leq 2m-\tau$.  
	 
	 If $y^{(0)} \leq 2m-\tau- 2$, then for any simple pair $(S,T)$,  
	 $$2a^{(0)}(S,T)+b^{(0)}(S,T)\leq 2x^{(0)} +y^{(0)} \le 2 + 2m-\tau-2 = 2m-\tau.$$ 
	 This implies that $(S,T)$ is not bad with respect to $(L,P^0)$. Hence  $\beta(P^0)=0$, in contrary to Claim \ref{second case half lemma}.   So we have $y^{(0)} \geq 2m-\tau-1$, thus $y^{(0)}= 2m-\tau-1$ or $y^{(0)}=2m-\tau$. 
		
	If $y^{(0)} = 2m-\tau-1$, then a bad simple pair with respect to $(L,P^0)$  consists  of the unique couple which is heavy for $P^0$ and the exactly $2m-\tau-1$ couples which are light for $P^0$. So
	 $\beta(P^0)=1$, in contrary to Claim \ref{second case half lemma}.

		Assume $y^{(0)} =2m-\tau$. Suppose $(S,T)$ is a bad simple pair with respect to $(L,P^2)$. By Claim \ref{claim-bad simple pair},  $$2a^{(2)}(S,T)+b^{(2)}(S,T) \geq \max\{2x^{(2)}+y^{(2)}+2m-\tau+1-\ell,2m-\tau+1\}.$$ 	 By Claim \ref{LuLv-HSL}, $x^{(2)}=y^{(0)}$ and $y^{(2)}=z^{(0)}=\ell-x^{(0)}-y^{(0)}=\ell-1-y^{(0)}$. Hence,  \begin{equation}
		\label{claim8-eq}
			2a^{(2)}(S,T)+b^{(2)}(S,T) \geq y^{(0)}+2m-\tau = 4m-2\tau.
		\end{equation}
		 As $a^{(2)}(S,T)+b^{(2)}(S,T)=2m-\tau-c^{(2)}(S,T) \leq 2m-\tau$, we have  $2a^{(2)}(S,T)+b^{(2)}(S,T) \leq 2x^{(2)}+(2m-\tau-x^{(2)}) = x^{(2)}+2m-\tau= 4m-2\tau$. Together with Inequality (\ref{claim8-eq}), we have $2a^{(2)}(S,T)+b^{(2)}(S,T)=4m-2\tau$ and hence $a^{(2)}(S,T)=x^{(2)}=2m-\tau$, i.e., a bad simple with respect to $(L,P^2)$ consists of exactly the $2m-\tau$ couples which are  heavy for $P^2$. So $\beta(P^2)=1$, in  contrary to Claim \ref{second case half lemma}. 
	\end{proof}

Without loss of generality, we assume that
\begin{itemize}
	\item $c_0c'_0$ and $c_1c'_1$ are  heavy for $P^0$ (and hence light for $P^1$ and safe for $P^2$ by Claim \ref{LuLv-HSL}).
	\item $c_2c'_2$ and $c_3c'_3$ are   light for $P^0$  (and hence  safe for $P^1$ heavy for $P^2$).
	\item  $c_4c'_4$ and $c_5c'_5$ are  safe for $P^0$ (and hence heavy for $P^1$ and light for $P^2$).
\end{itemize}   

 As $|V(P^0)|=1$, we know that $\hat{X}^0_1=\hat{X}^0_n=\emptyset$.
 	By Observation \ref{obs-couple},
      $c_0 \neq c'_0$ and $c_1 \neq c'_1$, and  we may   assume that $c_0 \notin $$\hat{X}^1_1 \cup \Lambda^1$.
      
       Let $S_1=\{c_0,c_4\}$, $T_1=\{c'_2,c'_4\}$,  and let $S_3=\{c_0,c_1,\ldots,c_5\}$, $T_3=\{c'_0,c'_1,\ldots,c'_5\}$. If $c_1\notin \hat{X}^1_1 \cup \Lambda^1$, then let $S_2=\{c_0,c_1,c_4,c_5\}$, $T_2=\{c'_2,c'_3,c'_4,c'_5\}$. Otherwise, 	 $c'_1\notin \hat{X}^1_n \cup \Lambda^1$ (again by Observation \ref{obs-couple}),  we let $S_2=\{c_0,c_2,c_4,c_5\}$, $T_2=\{c'_1,c'_3,c'_4,c'_5\}$.

 Now we show that for each $i\in \{0,1,2\}$, $j \in \{1,2,3\}$, $dam_{L,P^i}(S_j,T_j) \leq 2j$.
 
  For each  $i \in \{0,1,2\}$, among the six couples $c_0c'_0, \ldots, c_5c'_5$, two are light, two are safe and two are heavy  for $P^i$.   Therefore, $dam_{L,P^i}(S_3, T_3)=6$. 
  
 As $c_4c'_4$ is safe for $P^0$,
 $dam_{L,P^0}({c}_4, {c'}_4)=0$. Hence 
 $dam_{L,P^0}(S_1, T_1)=dam_{L,P^0}(c_0, c'_2) \le 2$. As $c_0\notin \hat{X}^1_1 \cup \Lambda^1$
 and $c_2c'_2$ is safe for $P^1$, we have 
  $dam_{L,P^1}(c_0, {c'}_2) =0$.  Hence 
  $dam_{L,P^1}(S_1, T_1)=dam_{L,P^1}({c}_4, {c'}_4) = 2$. As $c_0c'_0$ is safe for $P^2$ and $c_4c'_4$ is light for $P^2$, we have
   $dam_{L,P^2}({c}_0, {c'}_2) =1$ and $dam_{L,P^2}({c}_4, {c'}_4) =1$. Thus 
   $dam_{L,P^2}(S_1, T_1)=2$.
   
  Consider the case that  $c_1\notin \hat{X}^1_1 \cup \Lambda^1$. As $c_4c'_4,c_5c'_5$ are safe for $P^0$, we conclude that 
  $dam_{L,P^0}(S_2, T_2)=dam_{L,P^0}({c}_0, {c'}_2)+
 dam_{L,P^0}({c}_1, {c'}_3) \le 4$. As $c_0,c_1 \notin \hat{X}^1_1 \cup \Lambda^1$ and $c_2c'_2, c_3c'_3$ are safe for $P^1$, it follows that 
  $dam_{L,P^1}(S_2, T_2)=dam_{L,P^1}({c}_4, {c'}_4)+
  dam_{L,P^1}({c}_5, {c'}_5)  = 4$. As $c_0c'_0, c_1c'_1$ are safe for $P^2$, it follows that 
   $dam_{L,P^2}(c_0,c'_2)+ dam_{L,P^2}(c_1,c'_3) \le 2$. As $c_4c'_4, c_5c'_5$ are light for $P^2$, we have
   $dam_{L,P^2}(S_2, T_2) \le 2+ dam_{L,P^2}({c}_4, {c'}_4)+
  dam_{L,P^2}({c}_5, {c'}_5)  = 4$.
The other case is verified similarly and the details are omitted. 
	
	If $\tau \leq 2m-6$, then we know that $(S_3,T_3)$ is a simple pair of size $6$ which contradicts to Claim \ref{all at most light}. So we have that $\tau \in \{2m-4,2m-2\}$. 
	If $\tau =2m-4$,  then we let $S=S_2$ and $T=T_2$ and if $\tau=2m-2$, then we let $S=S_1$ and $T=T_1$. In each case,    $|S|=|T| = 2m-\tau$ and by the arguments above, $dam_{L, P^i}(S,T) \leq 2m-\tau$  for $i=0,1,2$.
	By \ref{C5}, for $i=0,1,2$,   $$dam_{L,P^i}(S,T) \leq 2m-\tau \leq   S_{L}(P^i)-2m|V(P^i)|.$$ 
	
	This completes the proof of Theorem \ref{main-second-part-stronger}.

\section{Proof of Lemma \ref{main-lemma}}
 
This section  proves Lemma \ref{main-lemma}.  I.e., 
\begin{equation}
	\label{orginal-def-lemma} 
	F(x,y)=\sum \binom{x}{a}\binom{y}{b}\binom{\ell-x-y}{2m-\tau-a-b} \le \frac 12 {\ell \choose 2m-\tau}-1,
\end{equation}
where $x+y \le \ell, 2x+y \le \ell+2m-\tau-1$, $m\geq 1$, $0 \leq \ell \leq 4m$, $0 \leq \tau \leq 2m-2$, $\ell +\tau \geq 2m+2$,   $\ell$ and $\tau$ are both even,
	and the summation is over  non-negative integer pairs $(a,b)$ for which $0 \leq a \leq x$, $0 \leq b \leq y$, $a+b \le 2m-\tau$ and  $2a+b \geq \max\{2x+y+2m-\tau+1-\ell, 2m-\tau+1\}$.

 Note that $a+b \le 2m-\tau$ and  $2a+b \geq  2m-\tau+1$ implies that $a \ge 1$.

     In the sequel, we define 
 	$$\binom{p}{q}_+=
 	\begin{cases}
 	\binom{p}{q} & \text{if $p \geq q \geq 0$,}\\
 	0 & \text{if $q<0$ or $p <q$.}
 	\end{cases}$$ 
For convenience, we allow $p<q$ or $q<0$ in the binormial coefficient in the summations below.   It is easy to check in these cases,  either the pair $(a,b)$ does not lie in the range of the summation, and hence  contributes $0$ to the summations, or by extending  the equality $\binom{p+1}{q}=\binom{p}{q}+\binom{p}{q-1}$ to $q=0$. For the readability, we suppress the index `$+$'.

 First, we analyze the monotonicity  of $F(x,y)$ about $y$ when $x$ is fixed, say $x=x_0$. For convenience, we let $2k = 2m-\tau$.
	\begin{lemma}
		\label{monotonous}
		Assume $x=x_0$ is fixed.      
		\begin{itemize}
			\item If $y \ge \ell-2x_0$, then $F(x_0,y+1)\leq F(x_0,y)$.
			\item If $ y < \ell -2x_0$, then $F(x_0,y) \leq F(x_0,y+1)$.
		\end{itemize} 
	\end{lemma}
	\begin{proof}
		In the following, let $z=\ell-x_0-y$.
		
		\noindent	
		{\bf Case 1.} $\ell-2x_0 \leq y \leq  \ell+ 2m-\tau-1-2x_0  = \ell+2k-1-2x_0$.
		
	In this case, $\max \{2x_0+y+2m-\tau+1-\ell, 2m-\tau+1\} = 2x_0+y+2m-\tau+1-\ell  =2x_0+y+2k+1-\ell$.
		For brevity, let
	$$
		 t(y) =   2x_0+y+1-\ell,  \ \  
	   r(y) =    t(y)+2k.
	$$
	 
		   As $2a+b \geq   r(y)$ and  $a+b \leq 2m-\tau =2k$, we have $a \geq r(y)-2k =t(y)$  and $  r(y) - 2a \le b \le 2k-a$.	 
		So    
		$$F(x_0,y)  
		 =\sum\limits_{i=t(y)}^{x_0}\sum\limits_{j=r(y)-2i}^{2k-i}\binom{x_0}{i}\binom{y}{j}\binom{z}{2k-i-j}.
		$$ 
		
		Note that $r(y+1)=r(y)+1$ and $t(y+1) = t(y)+1 $. Let $\Delta  = F(x_0,y) -F(x_0,y+1)$
		and for $t(y+1) \le i \le x_0$, let
	\begin{equation}
	\label{eq-lem9-1}
		\Delta_i=\sum\limits_{j=r(y)-2i}^{2k-i}\binom{y}{j}\binom{z}{2k-i-j}-\sum\limits_{j=r(y)+1-2i}^{2k-i}\binom{y+1}{j}\binom{z-1}{2k-i-j}.
	\end{equation}
	
	 Note that $2k-t(y)=r(y)-2t(y)$, we have 	
	 $$\Delta   \geq \sum\limits_{j=r(y)-2t(y)}^{2k-t(y)}\binom{x_0}{t(y)}\binom{y}{j}\binom{z}{2k-t(y)-j}+ \sum\limits_{i=t(y+1)}^{x_0}\binom{x_0}{i}\Delta_i = \binom{x_0}{t(y)}\binom{y}{2k-t(y)}+\sum\limits_{i=t(y+1)}^{x_0}\binom{x_0}{i}\Delta_i.$$
	Therefore, to prove that  $F(x_0,y) \ge F(x_0,y+1)$, it suffices to  prove that  $\Delta_i \geq 0$ for   $t(y)+1 \leq i \leq x_0$.  
 Using equalities $\binom{z}{2k-i-j} = \binom{z-1}{2k-i-j} + \binom{z-1}{2k-i-j-1}$ (in  the first sum of Equality (\ref{eq-lem9-1}))
	 and $\binom{y+1}{j} = \binom{y}{j}+\binom{y}{j-1}$ (in the second sum of Equality (\ref{eq-lem9-1})), and cancel the term $ \sum\limits_{j=r(y)+1-2i}^{ 2k-i}\binom{y}{j}\binom{z-1}{2k-i-j}$, we have
		$$
		\Delta_i 
	 = \sum\limits_{j=r(y)+1-2i}^{2k-i}\binom{y}{j}\binom{z-1}{ 2k-1-i-j}-\sum\limits_{j=r(y)+1-2i}^{ 2k-i}\binom{y}{j-1}\binom{z-1}{ 2k-i-j}+\binom{y}{r(y)-2i}\binom{z}{ 2k+i-r(y)}.$$
	When $j=2k-i$ in the first sum, we have $\binom{z-1}{-1}=0$. Writing the second sum  in the   equality as
	$\sum\limits_{j=r(y)-2i}^{ 2k-1-i}\binom{y}{j}\binom{z-1}{ 2k-1-i-j}$,   we have
	$$ \Delta_i = \binom{y}{r(y)-2i}\big[-\binom{z-1}{ 2k+i-r(y)-1}  + \binom{z}{2k+i-r(y)}\big]= \binom{y}{r(y)-2i}\binom{z-1}{ 2k+i-r(y)} \geq 0. $$

		\bigskip
		\noindent
		{\bf Case 2.} $ 0 \leq y < \ell -2x_0 $.
			
		In this case, $2x_0+y \leq \ell$, thus we have that $2a+b \geq \max\{2x_0+y+2k+1-\ell,2k+1\}=2k+1$. Let $s(y) = \max\{\lceil \frac{ 2k+1-y}{2}\rceil,1\}$. As $2a+b \ge  2k+1$ and $b \le y$,  we have $a \geq \lceil \frac{ 2k+1-y}{2} \rceil$. We have observed already that    $a\geq 1$. So $a \ge   s(y)$. and
		$$ F(x_0,y)  = \sum\limits_{i=s(y)}^{x_0}\sum\limits_{j= 2k+1-2i}^{ 2k-i}\binom{x_0}{i}\binom{y}{j}\binom{z}{ 2k-i-j}.$$
		
		Note that $s(y+1) \leq s(y)$ and the equality holds when $y$ is odd.
		Let $\Delta   = F(x_0,y+1) -F(x_0,y)$ and for $s(y) \le i \le x_0$, let 
	\begin{equation} 
		\label{second-deltai}
		\Delta_i  =\sum\limits_{j= 2k+1-2i}^{ 2k-i}\binom{y+1}{j}\binom{z-1}{ 2k-i-j}-\sum\limits_{j= 2k+1-2i}^{ 2k-i}\binom{y}{j}\binom{z}{2k-i-j}.
	\end{equation}
		Then 
		$\Delta  \geq   \sum\limits_{i=s(y)}^{x_0}\binom{x_0}{i}\Delta_i$.	To prove that  $F(x_0,y+1) \ge F(x_0,y)$,   it suffices to prove that for each $i$, $\Delta_i \geq 0$.
Using equalities $\binom{y+1}{j} = \binom{y}{j}+\binom{y}{j-1}$ and $\binom{z}{2k-i-j}=\binom{z-1}{2k-i-j}+\binom{z-1}{2k-i-j-1}$ in Equality (\ref{second-deltai})  and cancel the term $ \sum\limits_{j=2k+1-2i}^{ 2k-i}\binom{y}{j}\binom{z-1}{2k-i-j}$,  
 	we have 
$$\Delta_i =\sum\limits_{j=2k+1-2i}^{2k-i}\binom{y}{j-1}\binom{z-1}{2k-i-j}-\sum\limits_{j=2k+1-2i}^{2k-i}\binom{y}{j}\binom{z-1}{2k-1-i-j}.$$ 
When  $j=2k-i$ in the second sum, we have $\binom{z-1}{-1}=0$. Writing the first sum in the  equality above as
$\sum\limits_{j=2k-2i}^{ 2k-1-i}\binom{y}{j}\binom{z-1}{ 2k-1-i-j}$,    we have
$$ \Delta_i = \binom{y}{2k-2i}\binom{z-1}{i-1} \geq 0. $$
\end{proof}	
	 
Now, we continue with the proof of Lemma \ref{main-lemma}.  First assume that  $x < \frac{\ell}{2}$. By Lemma \ref{monotonous}, $F(x,y) \leq  F(x,\ell-2x)$. 
So it suffices to show that $F(x,\ell-2x) \le \frac 12 \binom{\ell}{2k}-1$.
Recall that (by Equality (\ref{orginal-def-lemma}))
 $$F(x,\ell-2x)  = \sum\limits_{t=2k+1}^{4k}\sum\limits_{2a+b =t}\binom{x}{a}\binom{\ell-2x}{b}\binom{x}{2k-a-b} =\sum\limits_{t=2k+1}^{4k} C(t,x),$$ 
where  $$C(t,x) =  \sum\limits_{2a + b=t}  \binom{x}{a} \binom{\ell-2x}{b} \binom{x}{2k -a-b} = \sum\limits_{2a \leq t}  \binom{x}{a} \binom{\ell-2x}{t-2a} \binom{x}{2k +a - t}.$$ 
Then $\sum_{t=0}^{4k}C(t,x) = \binom{\ell}{2k}$. Since for any $0 \leq t \leq 2k$,  $$	C(t,x)  = \sum\limits_{2a \leq t} \binom{x}{a} \binom{\ell-2x}{t-2a} \binom{x}{2k +a - t}  =  \sum\limits_{2a' \leq 4k-t} \binom{x}{a'} \binom{\ell-2x}{4k-t-a'} \binom{x}{a'+t-2k} = C(4k -t,x), $$
where $a'=2k+a-t$. So $$F(x,\ell-2x) = \sum_{t=2k+1}^{4k}C(t,x) = \frac {\binom{\ell}{2k} - C(2k,x)}{2} \le \frac 12 \binom{\ell}{2k}-1.$$ 
Here we used the fact that   $C(2k, x)\geq 2$ when $1 \leq x < \frac{\ell}{2}$. Indeed, if $x \geq k$, 
\begin{equation*} 
	C(2k, x) \geq \binom{x}{k}^2\binom{\ell-2x}{0}+\binom{x}{k-1}^2\binom{\ell-2x}{2} \geq 1+1 = 2. 
\end{equation*}
If $1 \leq x \leq k-1$, then 
\begin{equation*} 
C(2k, x) \geq \binom{x}{x}^2\binom{\ell-2x}{2k-2x}+\binom{x}{x-1}^2\binom{\ell-2x}{2k-2x+2} \geq 1+1 =2.
\end{equation*} 

Now we assume that $x \geq \frac{\ell}{2}$. It follows that  $y \geq 0\geq  \ell-2x$. Hence, by the first part of Lemma \ref{monotonous}, $F(x,y) \leq F(x,0)$. So it suffices to prove that $F(x,0) \leq \frac{1}{2}\binom{\ell}{2k}-1$.
	
	 Note that in this case,  $b=y=0$ and $2x+y+2k+1-\ell \geq  2k+1$, so $2a+b=2a \geq 2x+y+2k+1-\ell = 2x+2k+1-\ell$, which implies that $a \geq x+k-\ell'+1$, thus
	\begin{equation*} 
	F(x,0) = \sum\limits_{i=x+k-\ell'+1}^{2k}\binom{x}{i}\binom{\ell-x}{2k-i}. 
	\end{equation*}
We first prove that when $x \geq \frac{\ell}{2}=\ell'$, $F(x,0)> F(x+1,0)$. Let $\Delta  = F(x,0) -F(x+1,0)$, then
\begin{equation}
\label{delta3}
	\Delta=\sum\limits_{i=x+1+k-\ell'}^{2k}\binom{x}{i}\binom{\ell-x}{2k-i}-\sum\limits_{i=x+2+k-\ell'}^{2k}\binom{x+1}{i}\binom{\ell-1-x}{2k-i}.
\end{equation}

Using equalities $\binom{x+1}{i} = \binom{x}{i}+\binom{x}{i-1}$ and $\binom{\ell-x}{2k-i}=\binom{\ell-x-1}{2k-i}+\binom{\ell-x-1}{2k-i-1}$,  and cancel the term $ \sum\limits_{j=x+k+2-\ell'}^{2k-1}\binom{x}{i}\binom{\ell-x-1}{2k-i}$, we have

$$\Delta =\binom{x}{x+1+k-\ell'}\binom{\ell-x}{k+\ell'-x-1}+\sum\limits_{i=x+2+k-\ell'}^{ 2k}\binom{x}{i}\binom{\ell-1-x}{2k-1-i}  - \sum\limits_{i=x+2+k-\ell'}^{2k}\binom{x}{i-1}\binom{\ell-1-x}{2k-i}.$$
When $i=2k$ in the first sum above, we have $\binom{\ell-1-x}{-1}=0$. Writing the last sum in the  equality above as $\sum\limits_{i=x+1+k-\ell'}^{2k-1}\binom{x}{i}\binom{\ell-1-x}{2k-1-i}$,   we have 
\begin{eqnarray*}
\Delta  &=& \binom{x}{x+1+k-\ell'}\binom{\ell-x}{k+\ell'-x-1} - \binom{x}{x+1+k-\ell'}\binom{\ell-x-1}{k+\ell'-x-2} \\
&=&  \binom{x}{x+1+k-\ell'}\binom{\ell-x-1}{k+\ell'-x-1} \geq 0.
\end{eqnarray*}

So, $F(x,y) \leq F(\ell',0)$. Note that $
	\sum\limits_{i=k+1}^{2k}\binom{\ell'}{i}\binom{\ell'}{2k-i} = \sum\limits_{i=1}^{k-1}\binom{\ell'}{2k-i}\binom{\ell'}{i}$
	and 
	\begin{equation*} 
	\binom{\ell}{2k}  =\sum\limits_{i=k+1}^{2k}\binom{\ell'}{i}\binom{\ell'}{2k-i}+ \sum\limits_{i=1}^{k-1}\binom{\ell'}{2k-i}\binom{\ell'}{i}+\binom{\ell'}{k}^2.
\end{equation*}	
	As $\ell'=\frac{\ell}{2} \geq k+1 \geq 2$, we have    
	\begin{equation*}
	F(x,y) \leq F(\ell',0)
	= \sum\limits_{i=k+1}^{2k}\binom{\ell'}{i}\binom{\ell'}{2k-i}  =  \frac{\binom{\ell}{2k}-\binom{\ell'}{k}^2}{2} < \frac{1}{2}\binom{\ell}{2k}-1.
	\end{equation*}	
	
	 This completes the proof of Lemma \ref{main-lemma}. 

\end{document}